\documentclass[12pt]{article}
\usepackage[T1]{fontenc}
\usepackage[utf8]{inputenc}
\usepackage{lmodern}
\usepackage{amsmath,amssymb,amsthm,mathtools}
\usepackage{enumitem}
\usepackage[margin=1.15in]{geometry}
\usepackage[dvipsnames]{xcolor}
\usepackage{tikz}
\usetikzlibrary{arrows.meta,quotes}
\usepackage[colorlinks=true,linkcolor=blue,citecolor=blue,urlcolor=blue]{hyperref}
\providecommand{\U}[1]{\protect\rule{.1in}{.1in}}
\newcommand{\NN}{\mathbb{N}}
\newcommand{\ZZ}{\mathbb{Z}}

\newcommand{\Symm}{\mathfrak{S}}

\newcommand{\id}{\operatorname{id}}

\newcommand{\sgn}{\operatorname{sgn}}
\newcommand{\set}[1]{\left\{ #1 \right\}}
\newcommand{\abs}[1]{\left| #1 \right|}
\newcommand{\tup}[1]{\left( #1 \right)}
\newcommand{\ive}[1]{\left[ #1 \right]}
\theoremstyle{plain}
\newtheorem{theorem}{Theorem}[section]
\newtheorem{corollary}[theorem]{Corollary}
\newtheorem{proposition}[theorem]{Proposition}
\newtheorem{lemma}[theorem]{Lemma}
\theoremstyle{definition}
\newtheorem{definition}[theorem]{Definition}
\newtheorem{example}[theorem]{Example}
\newtheorem{remark}[theorem]{Remark}
\numberwithin{equation}{section}

\begin{document}

\title{Powers of matrices with all principal minors equal to $1$}
\author{Darij Grinberg, Hamesh M. Hamesh\thanks{This paper was written
by GPT-5.5 in June 2026 with some amount of strategic prompting.
It was then edited by DG to improve writing and generalize the results.
Later, GPT-5.6 Sol found relevant references and helped draw
Figure~\ref{fig.main-implications}.
\\
This work is in the public domain.}}
\date{1 August 2026}
\maketitle

\begin{abstract}
We say that a square matrix $A$ is \emph{$1$-principled} if all its
principal minors are equal to $1$.
We show that over any well-behaved ring, any power $A^m$
of a $1$-principled matrix $A$ is again $1$-principled.
Well-behaved rings include all reduced rings as well as
all quotients of commutative rings modulo integrally
closed ideals; in particular, all fields and all quotients of
$\mathbb{Z}$ are well-behaved.
We note that $m$ can be any integer, positive or negative.
This generalizes Problem B5 of the 2021 Putnam contest in multiple
directions.

Over arbitrary commutative rings, we identify a stronger property
that is always inherited by powers: We say that a matrix
$A = \left(a_{i,j}\right)_{i,j\in\left[n\right]}$
is \emph{$1$-nullcyclic} if all its diagonal entries are $1$ and
if all the cyclic products $a_{i_1, i_2} a_{i_2, i_3} \cdots a_{i_k, i_1}$
with $k>1$ and distinct $i_1,i_2,\ldots,i_k$ vanish.
We show that if $A$ is $1$-nullcyclic, then so is $A^m$ for any integer $m$.
Furthermore, every $1$-nullcyclic matrix is $1$-principled over
any commutative ring, while the converse holds if the ring is
well-behaved.

Along the way, we prove analogous results that don't require
the diagonal entries to be $1$. These are concerned with
\emph{principled matrices} (those whose principal minors
equal the respective products of diagonal entries) and
\emph{nullcyclic matrices} (those whose cyclic products
$a_{i_1, i_2} a_{i_2, i_3} \cdots a_{i_k, i_1}$
with $k>1$ and distinct $i_1,i_2,\ldots,i_k$ vanish);
their diagonal entries can be arbitrary.

A crucial auxiliary result, which holds for any $n\times n$-matrix $A$,
is that the cyclic products $a_{i_1, i_2} a_{i_2, i_3} \cdots a_{i_k, i_1}$
(with $k>1$ and distinct $i_1,i_2,\ldots,i_k$) are integral over the ideal
generated by the principal minors of $A$ minus the corresponding
products of diagonal entries of $A$.
\end{abstract}

\section{Definitions and the main result}

Throughout this note, rings are commutative, associative and unital. For
$n\in\mathbb{N}$, we set $\ive{n}:=\left\{  1,2,\ldots,n\right\}  $.

If $A=\tup{a_{i,j}}_{i,j\in\ive{n}}$ is an $n\times n$-matrix
over a ring $R$ and if $S\subseteq\ive{n}$, then $A_{S}$ denotes the
principal submatrix of $A$ with row and column set $S$. That is,
$A_{S}=\operatorname{sub}_{S}^{S}A$ in the notation of
\cite{GrinbergPrincipal}. For instance,
$\begin{pmatrix}
a & b & c\\
a^{\prime} & b^{\prime} & c^{\prime}\\
a^{\prime\prime} & b^{\prime\prime} & c^{\prime\prime}%
\end{pmatrix}_{\left\{  1,3\right\}  }
= \begin{pmatrix}
a & c\\
a^{\prime\prime} & c^{\prime\prime}
\end{pmatrix}$.
The \emph{principal minors} of an $n\times n$-matrix $A$ are its
minors $\det A_{S}$ for all $S\subseteq\ive{n}$. We
use the convention that the empty determinant is $1$.

Problem B5 of the 2021 Putnam contest \cite[Problem B5]{Putnam2021} asserts
that if $A\in\mathbb{Z}^{n\times n}$ is an integer matrix whose all principal
minors are odd, then all powers $A^{m}$ of $A$ have the same property. By
reducing the matrix modulo $2$, this can be restated as follows: If
$A\in\left(  \mathbb{Z}/2\right)  ^{n\times n}$ is a matrix over the
two-element field $\mathbb{Z}/2$ whose all principal minors equal $1$, then
all its powers $A^{m}$ have the same property. This suggests a generalization
to arbitrary commutative rings instead of $\mathbb{Z}/2$; however, this
generalization was disproved in \cite[\S 6]{GrinbergPrincipal} for $4\times
4$-matrices over a certain finite ring.\footnote{That said, a part of the
generalization is true over any $R$ (see \cite[Theorem 5.2]{GrinbergPrincipal}%
): If $A\in R^{n\times n}$ is a matrix whose all principal minors equal $1$,
then all \textbf{diagonal entries} of its powers $A^{m}$ equal $1$ as well
(even though some principal minors of $A^{m}$ may differ from $1$).}

In this note, we shall show that this generalization is nevertheless true if
we replace $\mathbb{Z}/2$ by any ring of the form $\mathbb{Z}/d$ with
$d \in \mathbb{Z}$. More generally, it is true over every quotient $D/I$
whose kernel ideal $I$ is integrally closed in $D$. This includes quotients of
Pr\"ufer domains by arbitrary ideals and quotients of normal domains by
principal ideals.

We introduce some terminology for the type of matrices we will study.

\begin{definition}
Let $R$ be a commutative ring, and let $A=\tup{a_{i,j}}
_{i,j\in\ive{n}}$ be an $n\times n$-matrix over $R$. We say that $A$ is
\emph{$1$-principled} if
\[
\det A_{S}=1 \qquad \text{for every }S\subseteq\ive{n}.
\]

\end{definition}

Note that the diagonal entries of an $n\times n$-matrix are its $1\times1$
principal minors. Hence, the diagonal entries of a $1$-principled matrix are
$1$.

Also, each $1$-principled matrix $A$ has determinant $\det A = 1$
(since $\det A = \det A_{\ive{n}}$ is a principal minor of $A$),
and thus is invertible (since a matrix with invertible determinant
is invertible). Thus, its powers $A^m$ are defined for all
integers $m$.

Our main result is the following theorem, which
recovers Problem B5 of the 2021 Putnam contest when we take
$D = \ZZ$ and $I = 2\ZZ$ and $m \geq 0$.

\begin{theorem}
\label{thm.main2} Let $D$ be a commutative ring, let $I$ be an integrally
closed ideal of $D$, and set $R:=D/I$. Let $A$ be a $1$-principled
matrix over $R$. Then $A^{m}$ is $1$-principled for every $m\in\mathbb{Z}$.
\end{theorem}

The notion of an integrally closed ideal will be recalled in
Section~\ref{sec.integral-closure}.

Our proof of Theorem~\ref{thm.main2} will be obtained by a combination
of several results involving other properties of matrices.
We shall introduce them one by one now.

\begin{definition}
An $n\times n$-matrix $A=\tup{a_{i,j}}_{i,j\in\ive{n}}$ is said
to be \emph{$1$-diagonal} if all its diagonal entries $a_{i,i}$
are $1$.
\end{definition}

Note that \cite[Theorem 5.2]{GrinbergPrincipal} says that if $A$
is a $1$-principled matrix and $m \in \NN$, then $A^m$ is
$1$-diagonal.

\begin{definition}
Let $R$ be a commutative ring, and let $A=\tup{a_{i,j}}
_{i,j\in\ive{n}}$ be an $n\times n$-matrix over $R$. We say that $A$ is
\emph{principled} if
\[
\det A_{S} = \prod_{i \in S} a_{i,i}
\qquad \text{for every }S\subseteq\ive{n}.
\]
\end{definition}

The following simple proposition is an exercise in unfolding
definitions:

\begin{proposition} \label{prop.1-princ}
Let $R$ be a commutative ring, and let $A=\tup{a_{i,j}}
_{i,j\in\ive{n}}$ be an $n\times n$-matrix over $R$.
Then, $A$ is $1$-principled if and only if $A$ is
$1$-diagonal and principled.
\end{proposition}

\begin{proof}
$\Longrightarrow:$ Assume that $A$ is $1$-principled.
Then, each diagonal entry $a_{i,i}$ of $A$ equals the
principal $1\times 1$-minor $\det A_{\set{i}}$, which
is $1$ since $A$ is $1$-principled. Thus, $A$ is
$1$-diagonal.
Furthermore, for every $S \subseteq \ive{n}$, we have
$\prod_{i \in S} a_{i,i} = 1$ (since all factors
$a_{i,i}$ of this product are $1$, as we just saw)
and $\det A_{S} = 1$ (since $A$ is $1$-principled), so
that $\det A_{S} = \prod_{i \in S} a_{i,i}$ (by
comparing the preceding two equalities). This shows that
$A$ is principled. Thus, we have proved that $A$ is
$1$-diagonal and principled.

$\Longleftarrow:$ Assume that $A$ is
$1$-diagonal and principled. Then, for every
$S \subseteq \ive{n}$, we have
$\det A_S = \prod_{i \in S} a_{i,i}$ (since $A$ is
principled). But all the factors $a_{i,i}$ of the
latter product are $1$ (since $A$ is $1$-diagonal), so
this equality rewrites as $\det A_S = \prod_{i \in S} 1
= 1$. This shows that $A$ is $1$-principled.
\end{proof}

Our last two matrix properties rely on a simple combinatorial
notion.

\begin{definition}
Let $S$ be a set. A \emph{cycle} on $S$ will mean a $k$-tuple%
\[
C=\left(  i_{1},i_{2},\ldots,i_{k}\right)
\]
where $k>0$ and where $i_{1},i_{2},\ldots,i_{k}$ are distinct elements of
$S$. To be more precise, the cycle is not this $k$-tuple itself, but
rather its equivalence class under cyclic rotation (i.e., we count
$\left(  i_{1},i_{2},\ldots,i_{k}\right)  $ and $\left(  i_{2},i_{3}%
,\ldots,i_{k},i_{1}\right)  $ as being the same cycle). This cycle is said
to have \emph{length} $k$, \emph{vertices} $i_1, i_2, \ldots, i_k$ and
\emph{arcs} $\tup{i_1, i_2},\ \tup{i_2, i_3},\ \ldots,\ \tup{i_k, i_1}$;
furthermore, we call it \emph{nontrivial} if $k>1$ (that is,
if the cycle has more than one arc).

\end{definition}

\begin{definition}
Let $A = \tup{a_{i,j}}_{i,j\in\ive{n}}$ be an $n\times n$-matrix
over a commutative ring $R$.

\begin{enumerate}
\item[\textbf{(a)}]
The \emph{$A$-weight} of a cycle $C=\left(  i_{1},i_{2},\ldots,i_{k}\right)$
on $\ive{n}$ is defined to be
\[
w_{A}(C):=a_{i_{1},i_{2}}a_{i_{2},i_{3}}\cdots a_{i_{k-1},i_{k}}a_{i_{k}%
,i_{1}}.
\]

\item[\textbf{(b)}] We say that $A$ is \emph{nullcyclic} if
\[
w_{A}(C)=0\qquad\text{for each nontrivial cycle }C\text{ on }\ive{n}
\]
(that is, each nontrivial cycle on $\ive{n}$ has $A$-weight $0$).

\item[\textbf{(c)}] We say that $A$ is \emph{$1$-nullcyclic} if
$A$ is $1$-diagonal and nullcyclic.
\end{enumerate}
\end{definition}

\begin{example}
\label{exa.unitri}
If a matrix $A = \tup{a_{i,j}}_{i,j\in\ive{n}}$
is triangular, then $A$ is nullcyclic.
Indeed, any nontrivial cycle on $\ive{n}$ has an arc $\tup{i,j}$
with $i<j$ and an arc $\tup{i,j}$ with $i>j$, and at least one
of these arcs will satisfy $a_{i,j} = 0$; thus, the $A$-weight
of the cycle is $0$.

Thus, if a matrix $A = \tup{a_{i,j}}_{i,j\in\ive{n}}$
is unitriangular (i.e., triangular and satisfies $a_{i,i} = 1$
for all $i \in \ive{n}$), then $A$ is $1$-nullcyclic.

However, there are $1$-nullcyclic matrices that are
not unitriangular, such as
\[
A = \begin{pmatrix} 1 & 2 \\ 2 & 1 \end{pmatrix}
\qquad \text{ over }  R = \ZZ / 4.
\]
\end{example}

We then claim the following twin to Theorem~\ref{thm.main2}
(but only for $m \geq 0$, since the invertibility of $A$
is no longer guaranteed):

\begin{theorem}
\label{thm.main} Let $D$ be a commutative ring, let $I$ be an integrally
closed ideal of $D$, and set $R:=D/I$. Let $A$ be a principled
matrix over $R$. Then $A^{m}$ is principled for every $m\in\mathbb{N}$.
\end{theorem}

The proof of Theorem~\ref{thm.main} relies on three independent
results, all of which hold over any (commutative) ring:
First, every nullcyclic matrix
is principled (Proposition~\ref{prop.strong-implies-principled}),
although the converse does not always hold.
Second, nullcyclic matrices are stable under powers
(Proposition~\ref{prop.strong-powers}).
Third, every nontrivial cycle's $A$-weight is integral over the
ideal generated by the principal-minor defects (i.e., the
principal minors minus the corresponding products
of diagonal entries) (Theorem~\ref{thm.cycle-weight-integral}).
In the setting of Theorem~\ref{thm.main}, this forces these weights to
belong to $I$, since $I$ is integrally closed.

The proof of Theorem~\ref{thm.main2} will reuse these auxiliary
results, partly in modified form.
Further results will be proved along the way.
Figure~\ref{fig.main-implications} provides a map of most of the
implications we will prove (for all $m\in\NN$): A solid arc stands for an
implication valid over every commutative ring. A dashed blue arc denotes an
implication for matrices over a quotient $D/I$, where $I$ is an integrally
closed ideal of a commutative ring $D$.
Labels on the arcs signify where the respective implications are shown
(the letters ``P.'' and ``T.'' stand for ``Proposition'' and
``Theorem'', respectively; trivial implications are unlabelled).
Some of the results on the $1$-side actually hold
for every $m\in\ZZ$, as their labels indicate.

After we prove all the results we have described, we shall discuss some
examples of integrally closed ideals.

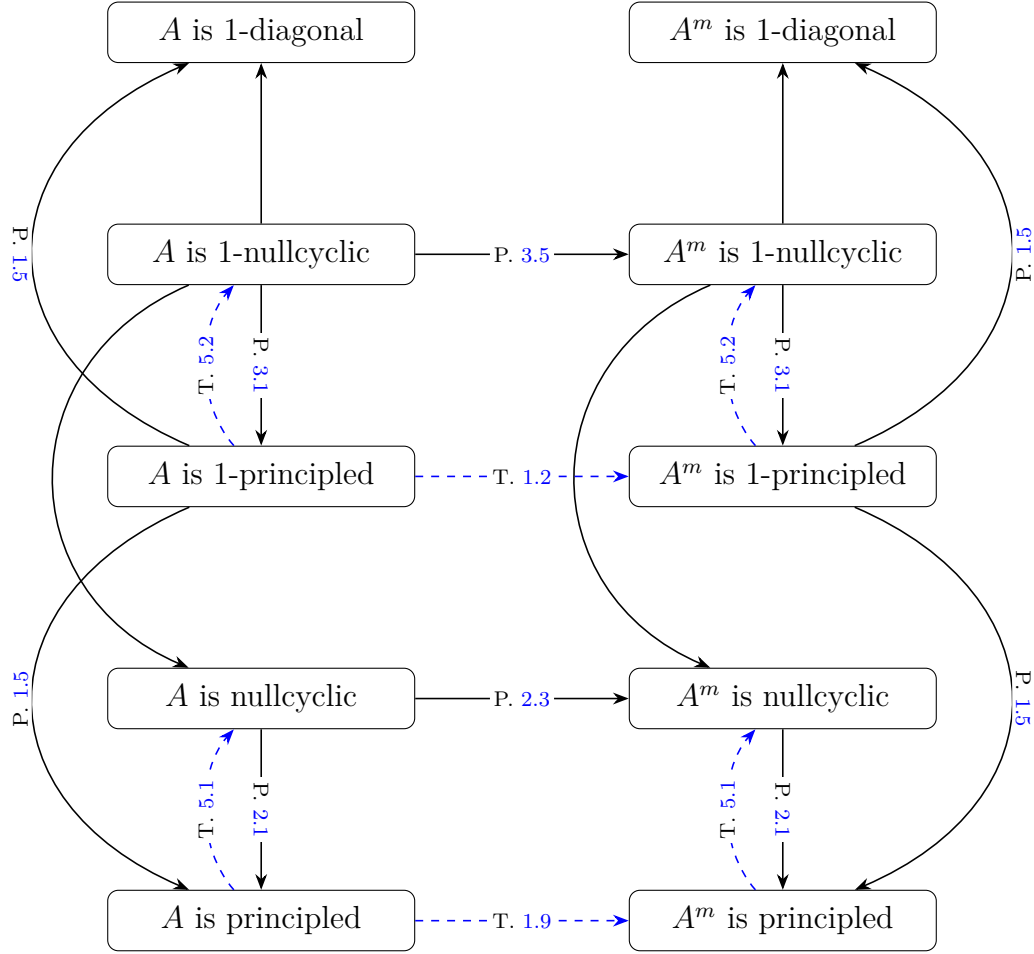
\begin{figure}[p]
\centering
\begin{tikzpicture}[
    x=6.9cm,
    y=2.35cm,
    statement/.style={draw, rounded corners, align=center,
      text width=3.85cm, minimum height=8mm, inner sep=3pt},
    implication/.style={-{Stealth[length=2.1mm]}, semithick},
    quotient/.style={implication, dashed, blue},
    every edge quotes/.style={font=\scriptsize, sloped, fill=white,
      inner sep=1.5pt, text=black},
    yscale=1.25
  ]
  \node[statement] (Ad)  at (0, 2) {$A$ is $1$-diagonal};
  \node[statement] (An1) at (0, 1) {$A$ is $1$-nullcyclic};
  \node[statement] (Ap1) at (0, 0) {$A$ is $1$-principled};
  \node[statement] (An)  at (0,-1) {$A$ is nullcyclic};
  \node[statement] (Ap)  at (0,-2) {$A$ is principled};

  \node[statement] (Amd)  at (1, 2) {$A^m$ is $1$-diagonal};
  \node[statement] (Amn1) at (1, 1) {$A^m$ is $1$-nullcyclic};
  \node[statement] (Amp1) at (1, 0) {$A^m$ is $1$-principled};
  \node[statement] (Amn)  at (1,-1) {$A^m$ is nullcyclic};
  \node[statement] (Amp)  at (1,-2) {$A^m$ is principled};

  \draw[implication] (An1) edge["P.~\ref{prop.strong-implies-principled2}"] (Ap1);
  \draw[implication] (An1) edge[] (Ad);
  \draw[implication] (Ap1) to[out=160, in=200, looseness=1.75]
    node[midway, below, sloped, font=\scriptsize, fill=white, inner sep=1.5pt]
    {P.~\ref{prop.1-princ}} (Ad);
  \draw[implication] (Ap1) to[out=200, in=160, looseness=1.75]
    node[midway, above, sloped, font=\scriptsize, fill=white, inner sep=1.5pt]
    {P.~\ref{prop.1-princ}} (Ap);
  \draw[implication] (An1) to[out=200, in=160, looseness=1.5] (An);
  \draw[implication] (An)  edge["P.~\ref{prop.strong-implies-principled}"'] (Ap);
  \draw[quotient]    (Ap1) edge[bend left=48,
    "T.~\ref{thm.integrally-closed-quotient-strong2}"] (An1);
  \draw[quotient]    (Ap) edge[bend left=48,
    "T.~\ref{thm.integrally-closed-quotient-strong}"'] (An);

  \draw[implication] (Amn1) edge["P.~\ref{prop.strong-implies-principled2}"] (Amp1);
  \draw[implication] (Amn1) edge[] (Amd);
  \draw[implication] (Amp1) to[out=20, in=-20, looseness=1.75]
    node[midway, below, sloped, font=\scriptsize, fill=white, inner sep=1.5pt]
    {P.~\ref{prop.1-princ}} (Amd);
  \draw[implication] (Amp1) to[out=-20, in=20, looseness=1.75]
    node[midway, above, sloped, font=\scriptsize, fill=white, inner sep=1.5pt]
    {P.~\ref{prop.1-princ}} (Amp);
  \draw[implication] (Amn1) to[out=200, in=160, looseness=1.5] (Amn);
  \draw[implication] (Amn)  edge["P.~\ref{prop.strong-implies-principled}"'] (Amp);
  \draw[quotient]    (Amp1) edge[bend left=48,
    "T.~\ref{thm.integrally-closed-quotient-strong2}"] (Amn1);
  \draw[quotient]    (Amp) edge[bend left=48,
    "T.~\ref{thm.integrally-closed-quotient-strong}"'] (Amn);

  \draw[implication] (An1) edge["P.~\ref{prop.strong-integer-powers2}"] (Amn1);
  \draw[quotient] (Ap1) edge["T.~\ref{thm.main2}"'] (Amp1);
  \draw[implication] (An) edge["P.~\ref{prop.strong-powers}"] (Amn);
  \draw[quotient] (Ap) edge["T.~\ref{thm.main}"'] (Amp);
\end{tikzpicture}
\caption{The main implications between the five matrix properties and their
power analogues.
Dashed blue arcs mean implications that hold when the base ring is
of the form $D/I$ for a ring $D$ and an integrally closed ideal $I$.
Solid black arcs are implications holding over all rings.
Not shown is the implication ``$A$ is $1$-principled''
$\Longrightarrow$ ``$A^m$ is $1$-diagonal'', proved in
\cite[Theorem 5.2]{GrinbergPrincipal} over all rings.}
\label{fig.main-implications}
\end{figure}
\clearpage

\begin{remark}
The condition of being $1$-principled has previously appeared in toric
topology.  A result of Masuda and Panov
\cite[Lemma~3.3]{MasudaPanov} (which they credit to Dobrinskaya)
says that, over an integral domain, every
$1$-principled matrix is conjugate by a permutation matrix to an upper
unitriangular matrix.  Thus, over an integral domain, all its powers are
$1$-principled for this stronger structural reason.  This triangularization
does not extend to arbitrary rings with zero divisors; for example,
the $1$-principled matrix in Example~\ref{exa.unitri} is not conjugate by a
permutation matrix to a triangular matrix.

In the special case of the field $\mathbb Z/2$, Choi
\cite[\S 2.1]{ChoiSmallCovers} used the above triangularization to identify
$1$-principled $n\times n$-matrices with labeled acyclic digraphs on $n$
vertices: the correspondence sends a digraph to its adjacency matrix plus
the identity matrix.  This arose in the study of small covers over cubes.

There is also a Gaussian-elimination interpretation.  Let $P$
be a permutation matrix, and let $\Delta_k$ denote the $k\times k$ leading
principal minor of $P^{-1}AP$, with $\Delta_0=1$.  The standard $LU$
criterion \cite[Theorem~9.1]{Higham} (easily generalized to arbitrary
commutative rings, see \cite[Proposition 4.3 (e)]{Sokal})
says that $P^{-1}AP$ has a unique
factorization $LU$, where $L$ is lower unitriangular and $U$ is upper
triangular with invertible entries on its diagonal,
whenever all the $\Delta_k$ are invertible; moreover, the $k$-th
diagonal entry of $U$ is $\Delta_k/\Delta_{k-1}$.  Hence, if $A$ is
$1$-principled, then $P^{-1}AP$ certainly does have such a factorization,
and moreover both $L$ and $U$ are unitriangular, for every $P$.
Conversely, if such a unitriangular $LU$ factorization exists for every $P$,
then all principal minors of $A$ are $1$.  Thus this property gives an
equivalent description of $1$-principled matrices, although this
description is not useful for proving any of our claims.
\end{remark}

\section{Nullcyclic matrices}

We begin with the theory of nullcyclic matrices.

\begin{proposition}
\label{prop.strong-implies-principled} Every nullcyclic matrix
over a commutative ring is principled.
\end{proposition}

\begin{proof}
Let $A=\tup{a_{i,j}}_{i,j\in\ive{n}}$ be a
nullcyclic $n\times n$-matrix. Let $S\subseteq\ive{n}$. We expand
\begin{equation}
\det A_{S}
= \sum_{\pi\in\Symm_S}\sgn\tup{\pi}
\prod_{i\in S}a_{i,\pi(i)}
\label{eq.lem.strong-implies-principled.1}
\end{equation}
(where $\Symm_S$ is the group of all permutations of $S$). The
identity permutation $\id \in \Symm_S$ contributes the addend
\[
\sgn\tup{\id} \prod_{i\in S}a_{i,\id(i)}
= \prod_{i\in S}a_{i,i}
\]
to the right-hand side of \eqref{eq.lem.strong-implies-principled.1}.
The remaining addends on the right-hand side of
\eqref{eq.lem.strong-implies-principled.1} correspond to
non-identity permutations $\pi \in \Symm_S$.
Every non-identity permutation $\pi \in \Symm_S$
has at least one nontrivial cycle $\left(  i_1, i_2, \ldots, i_k\right)$
in its cycle decomposition. The corresponding addend on the right-hand side of
\eqref{eq.lem.strong-implies-principled.1} therefore contains, as a factor,
the $A$-weight $a_{i_1,i_2} a_{i_2,i_3} \cdots a_{i_{k-1},i_k} a_{i_k,i_1}$
of this nontrivial cycle. This factor is $0$, since $A$ is
nullcyclic. Hence all non-identity addends on the right-hand side of
\eqref{eq.lem.strong-implies-principled.1} vanish, and so the equality
\eqref{eq.lem.strong-implies-principled.1} simplifies to
\[
\det A_{S}= \prod_{i\in S}a_{i,i}.
\]
That is, $A$ is principled.
\end{proof}

We shall use some basic graph theory (see, e.g., \cite[Chapter 4]{22s}). Let
$K_{n}^{\rightarrow}$ be the simple digraph (i.e., directed graph) with $n$
vertices $1,2,\ldots,n$ and $n^{2}$ arcs $\left(  i,j\right)  $ for all
$i,j\in\ive{n}$. What we called \textquotedblleft cycles on
$\ive{n}$\textquotedblright\ above are exactly the cycles of
$K_{n}^{\rightarrow}$ (considered up to cyclic
rotation). The nontrivial cycles on $\ive{n}$ are then the
cycles of $K_n^\to$ having length $>1$.
We make the following elementary observation about walks.

\begin{lemma}
\label{lem.closed-walk-has-cycle} Let
\[
i_{0}\rightarrow i_{1}\rightarrow\cdots\rightarrow i_{m}=i_{0}
\]
be a closed walk of a simple digraph. If this closed walk is not the
stationary walk $i_{0}\rightarrow i_{0}\rightarrow\cdots\rightarrow i_{0}$,
then it contains a nontrivial cycle. (\textquotedblleft
Contains\textquotedblright\ means that each arc of the cycle is an arc of the walk.)
\end{lemma}

\begin{proof}
Delete all loops\footnote{Recall that a \emph{loop} means an arc of the form
$\tup{v,v}$ for some vertex $v$.} from the closed walk. Since the original walk is not
stationary, some arcs remain upon this deletion.
Starting at any remaining arc and following the
walk cyclically, eventually a vertex is repeated. The part of the walk between
the first occurrence of this vertex and its next occurrence is a closed walk
with no repeated internal vertices. This closed walk is therefore a cycle,
and moreover a nontrivial cycle (since all loops have been deleted).
\end{proof}

\begin{proposition}
\label{prop.strong-powers} Let $A$ be a nullcyclic matrix over a
commutative ring $R$. Then $A^{m}$ is nullcyclic for every
$m\in\mathbb{N}$.
\end{proposition}

\begin{proof}
Let $m\in\mathbb{N}$. The case $m=0$ is clear, since $A^{0}=I_{n}$. Assume
$m>0$.

Write the $n\times n$-matrix $A$ as $A=\tup{a_{i,j}}_{i,j\in\ive{n}}$.
We also use $B_{i,j}$ to refer to the $\left(  i,j\right)  $-th entry of
any matrix $B$; thus, $A_{i,j}=a_{i,j}$ for all $i,j\in\ive{n}$.

It is well-known that for any two
vertices $i$ and $j$ of $K_{n}^{\rightarrow}$, we have%
\footnote{The formula \eqref{pf.lem.strong-powers.1}
(generalized to arbitrary digraphs)
appears, e.g., in \cite[last sentence of \S 1]{Zeil}
and in \cite[Theorem 4.7.1]{StanleyEC1}
(and also -- in an unweighted version -- in
\cite[Theorem 4.5.10]{22s} as well, where a more detailed proof
can be found, which can be easily adapted to the weighted case).
Anyway, the reader
can easily prove it by induction on $m$ using $A^m = A A^{m-1}$.}
\begin{equation}
\left(  A^{m}\right)_{i,j}
= \sum_{\substack{i=i_{0}\rightarrow i_{1}\rightarrow\cdots
\rightarrow i_{m}=j\\\text{is a walk of }K_{n}^{\rightarrow}}}
a_{i_{0},i_{1}}a_{i_{1},i_{2}}\cdots a_{i_{m-1},i_{m}}
.\label{pf.lem.strong-powers.1}
\end{equation}
Let us refer to the product $a_{i_{0},i_{1}}a_{i_{1},i_{2}}\cdots
a_{i_{m-1},i_{m}}$ in this sum as the \emph{$A$-weight} of the walk
$i=i_{0}\rightarrow i_{1}\rightarrow\cdots\rightarrow i_{m}=j$. Thus,
\eqref{pf.lem.strong-powers.1} says that
\begin{equation}
\left(  A^{m}\right)_{i,j}
= \tup{\text{sum of the $A$-weights of all length-$m$ walks from $i$ to $j$}}.
\qquad\qquad
\label{pf.lem.strong-powers.2}
\end{equation}

We must prove that $A^m$ is nullcyclic.
In other words,
we must prove that every nontrivial cycle on $\ive{n}$ has $A^{m}%
$-weight $0$. Let
\[
C=\left(  i_{1},i_{2},\ldots,i_{k}\right)
\]
be a nontrivial cycle on $\ive{n}$, with length $k>1$.
We must show that $w_{A^m} \tup{C} = 0$.

Note that $i_1 \neq i_2$ (by the definition of a cycle, since $k>1$).

Set $i_{k+1}=i_{1}$ (that is, read the indices cyclically modulo $k$).
Then,
\begin{align*}
w_{A^m} \tup{C}
&=
\left(  A^{m}\right)_{i_{1},i_{2}}\left(  A^{m}\right)_{i_{2},i_{3}}%
\cdots\left(  A^{m}\right)_{i_{k},i_{1}}
= \prod_{j=1}^{k}\left(
A^{m}\right)_{i_{j},i_{j+1}} \\
&= \prod_{j=1}^{k}
\tup{\text{sum of the $A$-weights of all length-$m$ walks from $i_j$ to $i_{j+1}$}}
\end{align*}
(by \eqref{pf.lem.strong-powers.2}, applied to $i_j$ and $i_{j+1}$
instead of $i$ and $j$).
Expanding this product, we obtain a
sum over all $k$-tuples $\left(W_1, W_2, \ldots, W_k\right)$
of length-$m$ walks from $i_{j}$ to $i_{j+1}$ for each
$j\in\ive{k}$. The addend corresponding to such a $k$-tuple is the
product of the $A$-weights of all these walks
$W_1, W_2, \ldots, W_k$; but this is, of course, the
$A$-weight of the closed walk $W_1W_2\cdots W_k$
(of length $km$) obtained by concatenating these
$k$ walks. This closed walk $W_1W_2\cdots W_k$ is not stationary
(since $i_1 \neq i_2$ are two distinct vertices on it).
Hence it contains a nontrivial cycle $C'$
(by Lemma~\ref{lem.closed-walk-has-cycle}). The $A$-weight of
this cycle $C'$ is $0$ since $A$ is nullcyclic.
Thus, the whole closed walk $W_1W_2\cdots W_k$ has $A$-weight $0$ as
well (since $C'$ is contained in $W_1W_2\cdots W_k$, and thus
the $A$-weight of $W_1W_2\cdots W_k$ is a multiple
of the $A$-weight of $C'$).

Thus, we have shown that $w_{A^m} \tup{C}$ is a sum of $A$-weights of certain
closed walks, but each of these closed walks has $A$-weight $0$. Hence,
\[
w_{A^m} \tup{C} = 0.
\]
As explained above, this completes the proof.
\end{proof}

\begin{corollary}
\label{cor.strong-powers-principled} Let $A$ be a nullcyclic
matrix over a commutative ring $R$. Then $A^{m}$ is principled for every
$m\in\mathbb{N}$.
\end{corollary}

\begin{proof}
The matrix $A^m$ is nullcyclic by
Proposition~\ref{prop.strong-powers}, and therefore is principled
by Proposition~\ref{prop.strong-implies-principled}.
\end{proof}

\begin{remark}
Proposition~\ref{prop.strong-powers} is genuinely a statement about powers,
not about products. Even commuting products of $1$-nullcyclic matrices
need not be nullcyclic.

For an explicit counterexample, let $R = \ZZ / 2 \ZZ = \mathbb F_2$. Let
\[
J=\begin{pmatrix}
1&1\\
1&1
\end{pmatrix},
\]
so that $J^2=0$. Now define the block matrices
\[
A=\begin{pmatrix}
I_2&0\\
J&I_2
\end{pmatrix}
=\begin{pmatrix}
1&0&0&0\\
0&1&0&0\\
1&1&1&0\\
1&1&0&1
\end{pmatrix},
\qquad
B=\begin{pmatrix}
I_2&J\\
0&I_2
\end{pmatrix}
=\begin{pmatrix}
1&0&1&1\\
0&1&1&1\\
0&0&1&0\\
0&0&0&1
\end{pmatrix}.
\]
Both $A$ and $B$ are unitriangular and thus $1$-nullcyclic
(by Example~\ref{exa.unitri}).

However,
\[
AB=\begin{pmatrix}
I_2&J\\
J&I_2+J^2
\end{pmatrix}
=\begin{pmatrix}
I_2&J\\
J&I_2
\end{pmatrix}
=\begin{pmatrix}
I_2+J^2&J\\
J&I_2
\end{pmatrix}=BA
=\begin{pmatrix}
1&0&1&1\\
0&1&1&1\\
1&1&1&0\\
1&1&0&1
\end{pmatrix}.
\]
This common product is not nullcyclic, since the nontrivial cycle
$\tup{1,3}$ has $AB$-weight $1$.
\end{remark}

\section{\texorpdfstring{$1$-Nullcyclic matrices}{1-Nullcyclic matrices}}

The analogue of Proposition~\ref{prop.strong-implies-principled}
for $1$-nullcyclic matrices follows easily from Proposition~\ref{prop.strong-implies-principled}:

\begin{proposition}
\label{prop.strong-implies-principled2} Every $1$-nullcyclic matrix
over a commutative ring is $1$-principled.
\end{proposition}

\begin{proof}
Let $A$ be a $1$-nullcyclic matrix. Then, $A$ is nullcyclic and
thus principled (by Proposition~\ref{prop.strong-implies-principled}).
But $A$ is furthermore $1$-diagonal (being $1$-nullcyclic).
Hence, $A$ is $1$-principled (by Proposition~\ref{prop.1-princ}).
\end{proof}

Next we derive the analogue of Proposition~\ref{prop.strong-powers}
for $1$-nullcyclic matrices:

\begin{proposition}
\label{prop.strong-powers2} Let $A$ be a $1$-nullcyclic matrix over a
commutative ring $R$. Then $A^{m}$ is $1$-nullcyclic for every
$m\in\mathbb{N}$.
\end{proposition}

\begin{proof}
Let $m\in\mathbb{N}$. The case $m=0$ is again clear, since $A^{0}=I_{n}$.
Assume $m>0$.

We know from Proposition~\ref{prop.strong-powers} that $A^m$ is
nullcyclic; thus, it remains to show that $A^m$ is $1$-diagonal.
In other words, we must show that all diagonal entries of $A^m$
are $1$.

Write the $n\times n$-matrix $A$ as $A=\tup{a_{i,j}}_{i,j\in\ive{n}}$.
We also use $B_{i,j}$ to refer to the $\left(  i,j\right)  $-th entry of
any matrix $B$; thus, $A_{i,j}=a_{i,j}$ for all $i,j\in\ive{n}$.

The matrix $A$ is $1$-nullcyclic, hence $1$-diagonal.
That is, all diagonal entries of $A$ are $1$: For
each $i\in\ive{n}$, we have
\begin{equation}
a_{i,i} = 1.
\label{pf.lem.strong-powers.aii}
\end{equation}

Fix $i \in \ive{n}$. The equality \eqref{pf.lem.strong-powers.2} shows that
$\left(  A^{m}\right)_{i,i}$ is the sum of the $A$-weights of
all length-$m$ walks from $i$ to $i$.
The stationary walk $i=i\rightarrow i\rightarrow\cdots\rightarrow i=i$
contributes $1$ to this sum, since its $A$-weight is
$a_{i,i}a_{i,i}\cdots a_{i,i}=a_{i,i}^{m}=1$ (by
\eqref{pf.lem.strong-powers.aii}).
Each of the other length-$m$ walks from $i$ to $i$ contains a nontrivial
cycle by Lemma~\ref{lem.closed-walk-has-cycle} (since it is a closed walk but
not stationary).
Since the latter cycle has $A$-weight $0$ (because $A$ is
nullcyclic), we conclude that the walk that contains it must have
$A$-weight $0$ as well (indeed, since $R$ is commutative, the $A$-weight of the
cycle is a factor of the $A$-weight of the walk).
Thus, $\left(  A^{m}\right)_{i,i}$ is the sum of a single $1$
(corresponding to the stationary walk
$i=i\rightarrow i\rightarrow\cdots\rightarrow i=i$) and a lot of $0$'s (coming
from all the other walks). Therefore,
\begin{equation}
\left(  A^{m}\right)_{i,i}=1.
\label{pf.lem.strong-powers.3}
\end{equation}

Forget that we fixed $i$. So we have proved \eqref{pf.lem.strong-powers.3} for
each $i \in \ive{n}$.
In other words, all diagonal entries of $A^m$ are $1$.
This completes the proof.
\end{proof}

\begin{remark}
Alternatively, we could have proved Proposition~\ref{prop.strong-powers2}
as follows:
Proposition~\ref{prop.strong-implies-principled2} shows that $A$
is $1$-principled. Thus,
\cite[Theorem 5.2]{GrinbergPrincipal} (which says that any
power of a $1$-principled matrix is $1$-diagonal) shows that $A^m$
is $1$-diagonal.
But Proposition~\ref{prop.strong-powers} shows that $A^m$ is
nullcyclic. Combining these two results, we conclude that $A^m$
is $1$-nullcyclic.
\end{remark}

We have shown in Proposition~\ref{prop.strong-powers2} that
$1$-nullcyclic matrices are stable under nonnegative powers. We now extend
this to all integer powers. The key observation is that a $1$-nullcyclic
matrix is unipotent with nilpotent off-diagonal part.

\begin{lemma}
\label{lem.1-nullcyclic-nilpotent-part}
Let $A=\tup{a_{i,j}}_{i,j\in\ive{n}}$ be a $1$-nullcyclic matrix over a
commutative ring $R$, and set $N:=A-I_n$. Then:

\begin{enumerate}

\item[\textbf{(a)}] Every closed walk
$i_{0}\rightarrow i_{1}\rightarrow\cdots\rightarrow i_{m}=i_{0}$ in $K_n^\to$
of length $m>0$ has $N$-weight\footnote{The $N$-weight of a
walk $i_{0}\rightarrow i_{1}\rightarrow\cdots\rightarrow i_{m}$
is defined to be $N_{i_0, i_1} N_{i_1, i_2} \cdots N_{i_{m-1}, i_m}$,
where $N_{i,j}$ denotes the $\tup{i,j}$-th entry of $N$.} $0$.

\item[\textbf{(b)}] We have $N^n = 0$.

\item[\textbf{(c)}] The matrix $A$ is invertible, and its inverse
$A^{-1}$ is $1$-nullcyclic.

\end{enumerate}

\end{lemma}

\begin{proof}
We will use the notation $B_{i,j}$ for the $\tup{i,j}$-th entry of
any matrix $B$.

The matrix $A$ is $1$-nullcyclic, thus $1$-diagonal.
That is, every diagonal entry $A_{i,i}$ of $A$ is $1$.
Hence, every diagonal entry $N_{i,i}$ of $N$ is $0$
(since $N = A - I_n$ and therefore $N_{i,i} = A_{i,i} - 1$).
On the other hand, for any $i \neq j$, we have
$N_{i,j} = A_{i,j}$ (since $N = A - I_n$ differs from $A$
only in its diagonal entries).
Thus, if a walk of $K_n^\to$ has no loops, then its
$N$-weight equals its $A$-weight.
In particular, any nontrivial cycle on $\ive{n}$ has
$N$-weight $0$ (because it has no loops, so its
$N$-weight equals its $A$-weight; but its $A$-weight is
$0$ because $A$ is nullcyclic).
\medskip

\textbf{(a)} Let
\[
i_0 \to i_1 \to \cdots \to i_m=i_0
\]
be a closed walk of positive length $m$. If this closed walk is stationary,
then all its arcs are loops, and its $N$-weight is $0$, since every diagonal
entry of $N$ is $0$.

If the closed walk is not stationary, then it contains a nontrivial cycle by
Lemma~\ref{lem.closed-walk-has-cycle}. The $N$-weight of this cycle is
$0$ (since any nontrivial cycle on $\ive{n}$ has
$N$-weight $0$). Hence the $N$-weight of the whole closed walk is
$0$ as well, since the $N$-weight of the cycle is a factor of it%
\footnote{Here we are using the commutativity of $R$.}.
This proves part \textbf{(a)}.
\medskip

\textbf{(b)} Let $i,j\in\ive{n}$. The walk expansion
\eqref{pf.lem.strong-powers.2} (applied to $N$ and $n$ instead of $A$
and $m$) expresses the
entry $\left(N^n\right)_{i,j}$ as the sum of the $N$-weights of all
length-$n$ walks
$i = i_{0}\rightarrow i_{1}\rightarrow\cdots\rightarrow i_n = j$
of $K_n^\to$ from $i$ to $j$. Every such walk has a repeated vertex
(by the pigeonhole principle, since $K_n^\to$ has only $n$ vertices),
and therefore contains a closed subwalk of positive length. By part
\textbf{(a)}, this closed subwalk has $N$-weight $0$.
Thus the whole walk has $N$-weight $0$ as well.
Since $\left(N^n\right)_{i,j}$ has been expressed as the sum of
the weights of such walks, we thus conclude that
$\left(N^n\right)_{i,j} = 0$.
Hence all entries of $N^{n}$ are $0$, proving $N^{n}=0$.
This proves part \textbf{(b)}.
\medskip

\textbf{(c)} By part \textbf{(b)}, we have $N^n=0$. Hence
$A=I_n+N$ is invertible, with inverse
\begin{equation}
A^{-1}=\sum_{r=0}^{n} \tup{-N}^r
\label{eq.lem.1-nullcyclic-nilpotent-part.inverse}
\end{equation}
(by the geometric series formula, here truncated to a finite
sum since $-N$ is nilpotent).
(We could replace the upper limit of this sum by $n-1$,
but we have no need for this.)

We first show that $A^{-1}$ is $1$-diagonal. Let $i\in\ive{n}$. From
\eqref{eq.lem.1-nullcyclic-nilpotent-part.inverse}, we obtain
\begin{align}
\tup{A^{-1}}_{i,i}
&= \left(\sum_{r=0}^{n} \tup{-N}^r\right)_{i,i}
= \sum_{r=0}^{n} \left(-1\right)^r \left(N^r\right)_{i,i}
\nonumber\\
&= 1+\sum_{r=1}^{n} \tup{-1}^r \left(N^r\right)_{i,i}.
\label{pf.lem.1-nullcyclic-nilpotent-part.c.1}
\end{align}
For each $r\geq 1$, the walk expansion \eqref{pf.lem.strong-powers.2}
(applied to $N$ and $r$ instead of $A$ and $m$)
expresses $\left(N^r\right)_{i,i}$ as a sum of $N$-weights of
positive-length closed walks. These weights are all $0$ by part
\textbf{(a)}. Hence, \eqref{pf.lem.1-nullcyclic-nilpotent-part.c.1}
simplifies to $\tup{A^{-1}}_{i,i} = 1$.
This shows that $A^{-1}$ is $1$-diagonal.

It remains to show that $A^{-1}$ is nullcyclic. Let
$C=\tup{i_1,i_2,\ldots,i_k}$ be a nontrivial cycle on $\ive{n}$, and set
$i_{k+1}:=i_1$.
We must show that $w_{A^{-1}}\tup{C} = 0$.

For each $j \in \ive{k}$, we can compare $\tup{i_j,i_{j+1}}$-entries in
\eqref{eq.lem.1-nullcyclic-nilpotent-part.inverse} to obtain
\begin{align*}
\tup{A^{-1}}_{i_j,i_{j+1}}
&= \left(\sum_{r=0}^{n} \tup{-N}^r\right)_{i_j,i_{j+1}}
= \sum_{r=0}^{n} \tup{-1}^r \left(N^r\right)_{i_j,i_{j+1}}
\\
&= \sum_{r=0}^{n} \tup{-1}^r \tup{\text{sum of the $N$-weights of all length-$r$ walks from $i_j$ to $i_{j+1}$}}
\end{align*}
(by \eqref{pf.lem.strong-powers.2}, applied to $N$ and $r$ instead of $A$ and $m$).
We can rewrite this further as
\begin{align}
\tup{A^{-1}}_{i_j,i_{j+1}}
= \tup{\text{sum of the signed $N$-weights of all length-$\leq n$ walks from $i_j$ to $i_{j+1}$}},
\nonumber
\end{align}
where the ``signed $N$-weight'' of a length-$r$ walk is defined to be $\tup{-1}^r$ times the $N$-weight of this walk.
Substituting this into
\[
w_{A^{-1}}\tup{C}
= \tup{A^{-1}}_{i_1,i_2} \tup{A^{-1}}_{i_2,i_3}\cdots \tup{A^{-1}}_{i_k,i_1}
= \prod_{j=1}^k \tup{A^{-1}}_{i_j,i_{j+1}},
\]
and then expanding the product, we express $w_{A^{-1}}\tup{C}$ as a sum
over $k$-tuples $\tup{W_1, W_2, \ldots, W_k}$ of walks from $i_j$ to $i_{j+1}$
for each $j \in \ive{k}$;
the addends of the sum are the products of the signed $N$-weights of these walks
$W_1, W_2, \ldots, W_k$.
As in the proof of Proposition~\ref{prop.strong-powers}, we can replace
these $k$-tuples $\tup{W_1, W_2, \ldots, W_k}$ by their concatenations
$W_1W_2\cdots W_k$, and then the addends are the signed $N$-weights of these
concatenations.
But each concatenation $W_1W_2\cdots W_k$ is a closed walk of positive length
(since it contains the two distinct vertices $i_1$ and $i_2$),
so its $N$-weight is $0$ by part \textbf{(a)}. Therefore all addends in the
expansion of $w_{A^{-1}}\tup{C}$ vanish, and thus $w_{A^{-1}}\tup{C}=0$.
This proves that
$A^{-1}$ is nullcyclic. Since $A^{-1}$ is also $1$-diagonal, we conclude that
$A^{-1}$ is $1$-nullcyclic.
This completes the proof of part \textbf{(c)}.
\end{proof}

\begin{proposition}
\label{prop.strong-integer-powers2}
Let $A$ be a $1$-nullcyclic matrix over a commutative ring $R$. Then $A^m$
is $1$-nullcyclic for every integer $m\in\ZZ$.
\end{proposition}

\begin{proof}
For $m\in\NN$, this is Proposition~\ref{prop.strong-powers2}. Thus, it
remains to treat negative $m$.

Let $m<0$. By Lemma~\ref{lem.1-nullcyclic-nilpotent-part} \textbf{(c)}, the
matrix $A^{-1}$ is $1$-nullcyclic. Also, $-m\in\NN$, and
$A^m=\tup{A^{-1}}^{-m}$. Hence Proposition~\ref{prop.strong-powers2},
applied to $A^{-1}$ and $-m$, shows that $A^m$ is $1$-nullcyclic.
\end{proof}

\begin{corollary}
\label{cor.strong-integer-powers-principled2}
Let $A$ be a $1$-nullcyclic matrix over a commutative ring $R$. Then $A^m$
is $1$-principled for every integer $m\in\ZZ$.
\end{corollary}

\begin{proof}
The matrix $A^m$ is $1$-nullcyclic by
Proposition~\ref{prop.strong-integer-powers2}, and therefore is
$1$-principled by Proposition~\ref{prop.strong-implies-principled2}.
\end{proof}

\section{Integral closure and cycle weights}\label{sec.integral-closure}

We now turn towards the harder, ``upstream'' direction, from
($1$-)principled to ($1$-)nullcyclic. As we already mentioned,
in general, a $1$-principled matrix is not always
nullcyclic; a counterexample is constructed in the footnote
in \cite[\S 6]{GrinbergPrincipal}\footnote{We recall this
counterexample:
If $F$ is any field, and if $R$ is the ring
$F\left[  x,y\right]  \diagup
\left(  x^3+y^3, xy, x^4, y^4 \right)$, and if
$A$ is the matrix
$\begin{pmatrix}
1 & 1 & 0 & 0\\
0 & 1 & y & x\\
x & 0 & 1 & y\\
y & 0 & x & 1
\end{pmatrix}$ over $R$, then $A$ is $1$-principled,
but $A^2$ is not, since $\det \left(A^2\right)_{\{2,3\}}
= 1 - x^3 \neq 1$ in $R$.
This counterexample is minimal in the sense that
there are no matrices of size $\leq 3$ that are
$1$-principled but have non-$1$-principled powers.}.
However, ($1$-)principledness implies a weaker version of
($1$-)nullcyclicity, which we will later leverage to obtain the
full property under certain conditions.

We will need the classical notion of \emph{integrality} over an ideal,
which we will now briefly recall;
see \cite{HunekeSwanson} for a much more extensive treatment.

\begin{definition}
Let $R$ be a commutative ring, let $I$ be an ideal of $R$, and let $x\in R$.
We say that $x$ is \emph{integral over $I$} if there exist a positive integer
$d$ and elements $c_j\in I^j$ for all $j\in\ive{d}$ such that
\[
x^d+c_1x^{d-1}+c_2x^{d-2}+\cdots+c_dx^0=0.
\]
The set of all elements of $R$ that are integral over $I$ is called the
\emph{integral closure} of $I$ and is denoted by $\overline I$. The ideal $I$
is said to be \emph{integrally closed} if $\overline I=I$.
\end{definition}

Thus, an element is integral over the zero ideal if and only if it is
nilpotent. We shall use the following standard properties of integral closure
of ideals.

\begin{lemma}
\label{lem.integral-closure-basics}
Let $R$ be a commutative ring, and let $I$ and $J$ be ideals of $R$. Then:
\begin{enumerate}
\item[\textbf{(a)}] The set $\overline I$ is an ideal of $R$ containing $I$
as a subset.

\item[\textbf{(b)}] If $I\subseteq J$, then $\overline I\subseteq\overline J$.

\item[\textbf{(c)}] The ideal $\overline I$ is integrally closed; that is,
$\overline{\overline I}=\overline I$.

\item[\textbf{(d)}] If $I\subseteq J\subseteq\overline I$, then
$\overline J=\overline I$.
\end{enumerate}
\end{lemma}

\begin{proof}
Parts \textbf{(a)} and \textbf{(c)} are
\cite[Corollary 1.3.1]{HunekeSwanson}.
Part \textbf{(b)} is trivial.
Part \textbf{(d)} follows from \textbf{(b)} and \textbf{(c)}: the
inclusions $I\subseteq J\subseteq\overline I$ yield
\[
\overline I\subseteq\overline J\subseteq
\overline{\overline I}=\overline I.
\qedhere
\]
\end{proof}

We next need a combinatorial ingredient. A \emph{Hamilton cycle on a
finite set $S$} means a cycle on $S$ whose vertices are all the elements
of $S$. In other words, it means a Hamilton
cycle of $K_S^\to$, where $K_S^\to$ is the digraph whose vertices are the
elements of $S$ and whose arcs are all the pairs $\tup{s,t}$ for
$s,t\in S$. We observe two obvious facts:
\begin{enumerate}
\item Each Hamilton cycle on a subset $S$ of $\ive{n}$
    is a cycle on $\ive{n}$.
    Conversely, each cycle on $\ive{n}$ is a Hamilton cycle on
    $S$, where $S$ is the set of vertices of this cycle.
\item A cycle on a finite set $S$ is Hamilton if and only
    if it has length $\abs{S}$.
\end{enumerate}

We identify each cycle with its set of arcs (since the latter set
uniquely determines the former cycle).
The \emph{multiset union} $H \uplus H'$ of two cycles $H$ and $H'$
is defined to be a multidigraph (i.e., a directed multigraph) whose multiset
of arcs is obtained by combining the sets of arcs of $H$ and of $H'$.
(If an arc appears in both $H$ and $H'$, then it will appear twice in this
multiset union.)

\begin{lemma}
\label{lem.two-Hamilton-cycles}
Let $S$ be a finite set, and let $H$ and $H'$ be two distinct Hamilton cycles
on $S$. Then the multiset union of (the sets of arcs of) $H$ and $H'$ can be
partitioned into at least three cycles, each having length strictly smaller
than $\left|S\right|$.
\end{lemma}

\begin{example}
Let $S=\left\{1,2,3,4,5\right\}$, and let
\[
H=\tup{1,2,3,4,5}
\qquad\text{and}\qquad
H'=\tup{1,3,5,2,4}
\]
be two Hamilton cycles on $S$. Their multiset union has arcs
\begin{align*}
\underbrace{(1,2),(2,3),(3,4),(4,5),(5,1)}_{\text{arcs of }H},
\underbrace{(1,3),(3,5),(5,2),(2,4),(4,1)}_{\text{arcs of }H'}.
\end{align*}
These ten arcs can be partitioned into the following three cycles:
\[
\tup{1,2,4},\qquad \tup{1,3,4,5},\qquad \tup{2,3,5}.
\]
The following two pictures show the same multidigraph $H\uplus H'$ twice.
In the left picture, the arcs of $H$ are red and the arcs of $H'$ are blue
and thick.
In the right picture, the three cycles just listed are drawn in different
colors and also different styles.
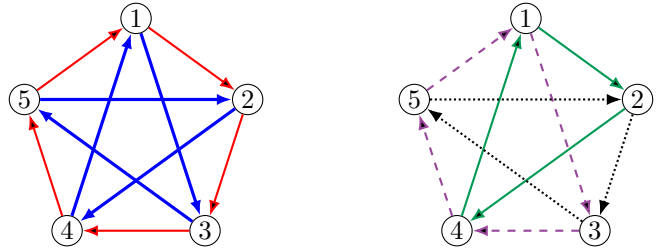
\begin{center}
\begin{tikzpicture}[scale=1.15,
  vertex/.style={circle,draw,fill=white,inner sep=1.2pt,font=\footnotesize},
  arr/.style={-{Latex[length=2mm]},line width=0.8pt}]
\foreach \name/\theta in {1/90,2/18,3/-54,4/-126,5/162}
  \node[vertex] (\name) at (\theta:1.35) {$\name$};
\foreach \u/\v in {1/2,2/3,3/4,4/5,5/1}
  \path[arr,draw=red] (\u) edge (\v);
\foreach \u/\v in {1/3,3/5,5/2,2/4,4/1}
  \path[arr,draw=blue, very thick] (\u) edge (\v);
\end{tikzpicture}
\qquad\qquad
\begin{tikzpicture}[scale=1.15,
  vertex/.style={circle,draw,fill=white,inner sep=1.2pt,font=\footnotesize},
  arr/.style={-{Latex[length=2mm]},line width=0.8pt}]
\foreach \name/\theta in {1/90,2/18,3/-54,4/-126,5/162}
  \node[vertex] (\name) at (\theta:1.35) {$\name$};
\foreach \u/\v in {1/2,2/4,4/1}
  \path[arr,draw=ForestGreen] (\u) edge (\v);
\foreach \u/\v in {1/3,3/4,4/5,5/1}
  \path[arr,draw=Purple, dashed] (\u) edge (\v);
\foreach \u/\v in {2/3,3/5,5/2}
  \path[arr,densely dotted] (\u) edge (\v);
\end{tikzpicture}
\end{center}
Indeed, these three cycles use the following arcs:
\begin{align*}
\tup{1,2,4} &: (1,2),(2,4),(4,1),\\
\tup{1,3,4,5} &: (1,3),(3,4),(4,5),(5,1),\\
\tup{2,3,5} &: (2,3),(3,5),(5,2).
\end{align*}
Thus each arc of $H\uplus H'$ is used exactly once. None of the three
cycles is Hamilton on $S$, since their lengths are $3,4,3$, respectively.
\end{example}

\begin{proof}[Proof of Lemma~\ref{lem.two-Hamilton-cycles}.]
Choose a vertex $v\in S$ whose outgoing arcs in $H$ and $H'$ are distinct.%
\footnote{Such a vertex must exist, since $H$ and $H'$ are distinct.}
Write these arcs as
\[
e=(v,a)\in H\qquad\text{and}\qquad e'=(v,b)\in H',
\]
where $a\neq b$. Let $P'$ be the directed path in $H'$ from $a$ back to $v$,
and let $P$ be the directed path in $H$ from $b$ back to $v$. Then the
walks
\begin{align*}
C &:= P'e \qquad \text{(that is, the path $P'$ followed by the arc $e$)}
\qquad \text{ and } \\
C' &:= Pe' \qquad \text{(that is, the path $P$ followed by the arc $e'$)}
\end{align*}
are cycles. Both have length at most $\left|S\right| - 1$. Indeed,
if the cycle $C$ had length $\geq \left|S\right|$, then it would be
Hamilton, and thus the path $P'$ would contain every vertex.
Hence the complementary path in $H'$ from $v$ to $a$ would consist
of a single arc. This arc would be the outgoing arc $(v,b)$ of $v$ in $H'$,
forcing $a=b$. The same argument applies to $C'$.

The cycles $C$ and $C'$ are arc-disjoint in the multidigraph
$H\uplus H'$.
Indeed:
\begin{enumerate}
\item The only $H$-arc used by $C$ is $e$, while the $H$-arcs
      used by $C'$ lie in $P$, which does not use $e$ because
      $P$ ends at $v$. Thus, $C$ and $C'$ have no $H$-arcs in
      common.
\item Similarly, $C$ and $C'$ have no $H'$-arcs in common.
\end{enumerate}

Remove the arcs of the two cycles $C$ and $C'$ from this multidigraph
$H\uplus H'$.
The remaining multidigraph is balanced\footnote{A
multidigraph is said to be \emph{balanced} if for each vertex $v$, the
indegree of $v$ equals the outdegree of $v$.}, because both the
original multidigraph $H\uplus H'$ and the removed union $C\cup C'$ are balanced.
Hence its arcs can be partitioned into cycles\footnote{We are using
a well-known result saying that the multiset of arcs of a balanced multidigraph
can be partitioned into cycles. This can be proved in many ways, e.g.:
Start walking at any non-isolated vertex; each time you enter a vertex,
the balancedness will ensure that you will be able to exit again; sooner
or later you will run into a cycle. Whenever this happens, remove the cycle
from the digraph, and repeat the same procedure. Removing a cycle leaves
the digraph balanced, so this algorithm will continue until the digraph
has no arcs left; at that point, the cycles obtained will form a partition
of the set of all arcs.}. None of these new cycles contains $v$, since
both outgoing arcs from $v$ in $H\uplus H'$
have been removed when we removed the arcs of $C$ and $C'$. Thus
every new cycle has length at most $\left|S\right| - 1$.

We have therefore partitioned the $2\left|S\right|$ arcs of
$H\uplus H'$ into cycles of length at most
$\left|S\right|-1$. Consequently, the number of these cycles
is at least
\[
\frac{2\left|S\right|}{\left|S\right|-1} > 2,
\]
that is, at least $3$ (since it is an integer).
This proves the lemma.
\end{proof}

Let $A=\tup{a_{i,j}}_{i,j\in\ive{n}}$ be a matrix over a commutative ring
$R$. For every integer $r\geq2$, let $K_{<r}(A)$ denote the ideal of $R$
generated by the $A$-weights of all nontrivial cycles on $\ive{n}$
whose length is strictly smaller than $r$.
Thus, $K_{<2}(A)=0$ (since a nontrivial cycle cannot have length
smaller than $2$).

\begin{lemma}
\label{lem.products-Hamilton-cycles}
Let $S\subseteq\ive{n}$ have size $r\geq2$, and let
$H_1,H_2,\ldots,H_t$ be $t$ distinct Hamilton cycles on $S$, where $t\geq2$.
Then
\[
w_A(H_1)w_A(H_2)\cdots w_A(H_t)\in K_{<r}(A)^t.
\]
\end{lemma}

\begin{proof}
Pick any $i \neq j$ in $\ive{t}$.
By Lemma~\ref{lem.two-Hamilton-cycles}, the multiset union of the two distinct
Hamilton cycles $H_i$ and $H_j$ on $S$ can be partitioned into at least three
cycles of length strictly smaller than $r$. These latter cycles are all nontrivial,
since they cannot contain any loops (indeed, all their arcs must be arcs of the
original two Hamilton cycles, but those did not contain any loops).
Hence, their $A$-weights belong to $K_{<r}(A)$.
Since the product of the $A$-weights is unchanged by repartitioning the same
multiset of arcs, this shows that
\begin{align}
w_A(H_i)w_A(H_j)\in K_{<r}(A)^3.
\label{pf.lem.products-Hamilton-cycles.4}
\end{align}
Thus, we have proved \eqref{pf.lem.products-Hamilton-cycles.4} for any $i \neq j$
in $\ive{t}$.

Now, pair off $2\left\lfloor t/2\right\rfloor$ of the $t$ Hamilton cycles
$H_1,H_2,\ldots,H_t$. Multiplying
the inclusions \eqref{pf.lem.products-Hamilton-cycles.4} for these pairs,
and multiplying by the $A$-weight of the remaining cycle if $t$ is odd, gives
\[
w_A(H_1)w_A(H_2)\cdots w_A(H_t)
\in K_{<r}(A)^{3\left\lfloor t/2\right\rfloor}.
\]
Since $3\left\lfloor t/2\right\rfloor\geq t$ for every $t\geq2$, the right
hand side is contained in $K_{<r}(A)^t$.
\end{proof}

We can now state the universal algebraic result.
For any matrix $A=\tup{a_{i,j}}_{i,j\in\ive{n}}$ over any
commutative ring $R$, we define the
\emph{principal-minor-defect ideal} of $A$ to be the ideal
\[
J(A):=\left(\det A_S-\prod_{i\in S}a_{i,i}
\ \middle|\ S\subseteq\ive{n}\right)
\subseteq R
\]
(that is, the ideal of $R$ generated by all $2^n$ differences
$\det A_S - \prod_{i\in S}a_{i,i}$,
where $S$ ranges over the subsets of $\ive{n}$).
Note that the generators $\det A_S-\prod_{i\in S}a_{i,i}$
of this ideal for subsets $S$ satisfying $\abs{S}\leq 1$
are zero, since the determinant of a $1\times 1$-matrix or
of a $0\times 0$-matrix is the product of its diagonal
entries.

\begin{theorem}
\label{thm.cycle-weight-integral}
Let $A$ be an $n\times n$-matrix over a commutative ring $R$.
Then the $A$-weight of every
nontrivial cycle on $\ive{n}$ is integral over $J(A)$. Equivalently,
every nontrivial cycle $C$ on $\ive{n}$ satisfies
\[
w_A(C)\in\overline{J(A)}.
\]
\end{theorem}

\begin{proof}
We must show that each nontrivial cycle $H$ on $\ive{n}$
satisfies $w_A\tup{H} \in \overline{J(A)}$.
In other words, we must show that for each $r\geq 2$ and each
$r$-element subset $S$ of $\ive{n}$, each Hamilton cycle $H$ on $S$
satisfies $w_A\tup{H} \in \overline{J(A)}$
(since any nontrivial cycle is a Hamilton cycle on a subset of
size $\geq 2$).

We use strong induction on $r$.
Fix an $r$-element subset
$S\subseteq\ive{n}$, where $r\geq2$, and let $\mathcal H_S$ be the set of all
Hamilton cycles on $S$. For every $H\in\mathcal H_S$, set
\[
z_H:=(-1)^{r-1}w_A(H).
\]
The sign $(-1)^{r-1}$ is the sign of the cyclic permutation corresponding to
$H$ (that is, of the permutation of $S$ that sends each vertex of $H$
to the next vertex that follows it on $H$).

Set
\[
K:=K_{<r}(A)
\qquad\text{and}\qquad
L:=J(A)+K.
\]
We first show that each $z_H$ is integral over $L$. Let
$e_t$ be the $t$-th elementary symmetric polynomial in the elements
$z_H$ for $H\in\mathcal H_S$.
In particular, $e_0 = 1$ and $e_1 = \sum_{H\in\mathcal H_S}z_H$.

We shall now use the determinant expansion
\eqref{eq.lem.strong-implies-principled.1} to show that
\begin{equation}
e_1 \in L.
\label{eq.Hamilton-e1-in-L}
\end{equation}

\begin{proof}[Proof of \eqref{eq.Hamilton-e1-in-L}.]
Each addend on the right-hand side of
\eqref{eq.lem.strong-implies-principled.1}
corresponds to some permutation $\pi \in \Symm_S$.
We can thus classify these addends into three classes:
\begin{enumerate}
\item[Class 1] consists of the single addend for the
      permutation $\pi = \id \in \Symm_S$.
      This addend is
      \[
      \sgn(\id) \prod_{i\in S}a_{i,\id(i)} = \prod_{i\in S}a_{i,i}.
      \]
\item[Class 2] consists of the addends corresponding
      to permutations $\pi$ that are $r$-cycles.
      Such permutations $\pi$ are in a one-to-one correspondence
      with the Hamilton cycles $H \in \mathcal{H}_S$
      (indeed, each such permutation $\pi$ gives rise to the
      Hamilton cycle
      $\tup{\pi^0(s), \pi^1(s), \ldots, \pi^{r-1}(s)}$ for an
      arbitrary $s\in S$; the choice of $s$ is immaterial).
      The addend corresponding to a given
      Hamilton cycle $H \in \mathcal{H}_S$ is
      \[
      \sgn\tup{\pi} \prod_{i\in S} a_{i, \pi(i)}
      = \tup{-1}^{r-1} \prod_{k=0}^{r-1} a_{\pi^k(s), \pi^{k+1}(s)}
      = \tup{-1}^{r-1} w_A\tup{H} = z_H.
      \]
\item[Class 3] consists of the addends corresponding to
      permutations $\pi \neq \id$ that have at least two
      cycles.
      For such a permutation $\pi$, all cycles of $\pi$ must
      have length smaller than $r$ (since a cycle of
      length $\geq r$ would cover all elements of $S$, thus
      leaving no room for a second cycle), and at least one
      of these cycles must be nontrivial (because $\pi \neq \id$).
      Therefore, the corresponding addend contains the $A$-weight
      of a nontrivial cycle of length strictly smaller than $r$
      as a factor; hence, this addend belongs to
      $K_{<r}(A) = K \subseteq L$, thus is $\equiv 0 \mod L$.
\end{enumerate}
Hence, reduced modulo $L$, the equality
\eqref{eq.lem.strong-implies-principled.1} becomes
\begin{align*}
\det A_{S}
&\equiv
\underbrace{\prod_{i\in S}a_{i,i}}_{\text{sum of addends of Class 1}}
+ \underbrace{\sum_{H \in \mathcal{H}_S} z_H}_{\text{sum of addends of Class 2}}
+ \underbrace{0}_{\text{sum of addends of Class 3}} \\
&= \prod_{i\in S}a_{i,i} + \sum_{H \in \mathcal{H}_S} z_H \mod L.
\end{align*}
Therefore,
\[
\sum_{H \in \mathcal{H}_S} z_H \equiv \det A_S - \prod_{i\in S}a_{i,i} \equiv 0 \mod L
\]
(since the definition of $J(A)$ yields
$\det A_{S} - \prod_{i\in S}a_{i,i} \in J(A) \subseteq L$).
That is, $\sum_{H \in \mathcal{H}_S} z_H \in L$.
In other words, $e_1 \in L$
(since $e_1 = \sum_{H \in \mathcal{H}_S} z_H$).
This proves \eqref{eq.Hamilton-e1-in-L}.
\end{proof}

Now we shall generalize \eqref{eq.Hamilton-e1-in-L}
by showing that
\begin{align}
e_t \in L^t \qquad \text{ for each } t \geq 0.
\label{eq.Hamilton-et-in-Lt}
\end{align}

\begin{proof}[Proof of \eqref{eq.Hamilton-et-in-Lt}.]
If $t = 0$, then this is obvious (since $L^0 = R$).
If $t = 1$, then it follows from \eqref{eq.Hamilton-e1-in-L}.
Thus, assume that $t \geq 2$ henceforth.
Now, $e_t$ is defined as the sum of the $t$-wise products
of the $z_H$'s with $H \in \mathcal{H}_S$.
Each addend in this sum is, up to sign, a product of the
$A$-weights of $t$ distinct Hamilton cycles on $S$
(since the $z_H$'s are, up to sign, the $A$-weights of these
cycles).
But Lemma~\ref{lem.products-Hamilton-cycles} shows that
each such product belongs to $K_{<r}(A)^t$.
Therefore, their sum $e_t$ belongs to $K_{<r}(A)^t$ as well,
and thus also to $L^t$
(since $K_{<r}(A) = K \subseteq L$).
This proves \eqref{eq.Hamilton-et-in-Lt}.
\end{proof}

But Vi\`ete's formulas show that every $z_H$ is a root of
the monic polynomial
\[
\prod_{G\in\mathcal H_S}(X-z_G)
=X^{\left|\mathcal H_S\right|}-e_1X^{\left|\mathcal H_S\right|-1}
+e_2X^{\left|\mathcal H_S\right|-2}-\cdots
\in R\left[X\right],
\]
whose $X^{\left|\mathcal H_S\right|-t}$-coefficient belongs to $L^t$
(by \eqref{eq.Hamilton-et-in-Lt}).
Thus every $z_H$ is integral over $L$.
Therefore, every $w_A(H)$ is integral over $L$ as well (since
$z_H$ is $w_A(H)$ up to sign).
In other words,
\begin{align}
w_A(H) \in \overline L
\qquad \text{ for each }
H \in \mathcal{H}_S.
\label{eq.wAH-in-L1}
\end{align}

By the induction hypothesis, the $A$-weights of all nontrivial cycles of
length strictly smaller than $r$ belong to $\overline{J(A)}$. Since
$\overline{J(A)}$ is an ideal, this yields
\[
K\subseteq\overline{J(A)}
\]
(since $K = K_{<r}(A)$ is the ideal generated by these $A$-weights).
Hence
$L = J(A)+K \subseteq \overline{J(A)}$
(since $J(A) \subseteq \overline{J(A)}$ and $K\subseteq\overline{J(A)}$).
Thus,
\[
J(A) \subseteq L \subseteq \overline{J(A)}.
\]
Hence, Lemma~\ref{lem.integral-closure-basics} \textbf{(d)} now
gives $\overline{L} = \overline{J(A)}$.
Thus, \eqref{eq.wAH-in-L1} rewrites as
\begin{align*}
w_A(H) \in \overline{J(A)}
\qquad \text{ for each }
H \in \mathcal{H}_S.
\end{align*}
In other words, each Hamilton cycle $H$ on $S$ satisfies
$w_A(H) \in \overline{J(A)}$.
This completes the induction.
\end{proof}

\begin{remark}
The relation between principal minors and cycle weights has an
algebro-geometric antecedent in the work of Lin and Sturmfels
\cite[Proposition~4, Lemma~6 and Proposition~7]{LinSturmfels}.  Over
$\mathbb C$, they describe the change of coordinates between principal
minors and cycle-sums, give a cycle decomposition closely related to
Lemma~\ref{lem.two-Hamilton-cycles}, and prove that the algebra generated by
the individual cycle monomials is integral over the algebra generated by the
principal minors.  In particular, \cite[Lemma~6]{LinSturmfels} is a
stronger version of our Lemma~\ref{lem.two-Hamilton-cycles},
whereas \cite[Proposition~7]{LinSturmfels} is an analogue of
our Theorem~\ref{thm.cycle-weight-integral} for integrality over
rings rather than over ideals.
\end{remark}

\begin{corollary}
\label{cor.principled-cycle-nilpotent}
Let $A$ be a principled $n\times n$-matrix over a
commutative ring $R$. Then the
$A$-weight of every nontrivial cycle on $\ive{n}$ is nilpotent.
\end{corollary}

\begin{proof}
Since $A$ is principled, each subset $S$ of $\ive{n}$ satisfies
$\det A_S = \prod_{i\in S}a_{i,i}$ and therefore
$\det A_S-\prod_{i\in S}a_{i,i} = 0$.
Thus, $J(A)=0$ (since we have just shown that all
generators of $J(A)$ are $0$).
Hence, Theorem~\ref{thm.cycle-weight-integral} shows that
every nontrivial cycle's $A$-weight is integral over the
zero ideal, and therefore nilpotent.
\end{proof}

\begin{corollary}
\label{cor.reduced-principled-strong}
Over a reduced commutative ring, every principled matrix is
nullcyclic. Consequently, every power of a principled matrix over a
reduced ring is principled.
\end{corollary}

\begin{proof}
A reduced ring has no nonzero nilpotent elements, so the first claim follows
from Corollary~\ref{cor.principled-cycle-nilpotent}. The second then follows
from Corollary~\ref{cor.strong-powers-principled}.
\end{proof}

\begin{corollary}
\label{cor.reduced-principled-strong2}
Over a reduced commutative ring, every $1$-principled matrix is
$1$-nullcyclic. Consequently, every power of a $1$-principled matrix over a
reduced ring is $1$-principled.
\end{corollary}

\begin{proof}
Every $1$-principled matrix is principled (by
Proposition~\ref{prop.1-princ}). Hence,
the first claim of the corollary follows from the first claim of
Corollary~\ref{cor.reduced-principled-strong}, combined with the fact
that every $1$-principled matrix is $1$-diagonal (see
Proposition~\ref{prop.1-princ}).
The second claim then follows
from Corollary~\ref{cor.strong-integer-powers-principled2}.
\end{proof}

\section{Quotients by integrally closed ideals}

\subsection{The general case}

The universal theorems from the preceding section have the following immediate
consequence.

\begin{theorem}
\label{thm.integrally-closed-quotient-strong}
Let $D$ be a commutative ring, let $I$ be an integrally closed ideal of
$D$, and set $R:=D/I$. Let $A$ be a principled matrix over $R$. Then $A$
is nullcyclic.
\end{theorem}

\begin{proof}
Choose a matrix $\widetilde A = \tup{\widetilde a_{i,j}}_{i, j \in \ive{n}}$
over $D$ whose image modulo $I$ is $A$. Since
$A$ is principled, we have
\[
\det \widetilde A_S - \prod_{i\in S} \widetilde a_{i,i} \in I
\qquad\text{for every }S\subseteq\ive{n}.
\]
Thus, $J(\widetilde A) \subseteq I$. By
Theorem~\ref{thm.cycle-weight-integral}, every nontrivial cycle
$C$ on $\ive{n}$ satisfies
$w_{\widetilde A}(C) \in \overline{J(\widetilde A)} \subseteq
\overline{I}$ (by Lemma~\ref{lem.integral-closure-basics} (b),
since $J(\widetilde A)\subseteq I$)
and therefore $w_{\widetilde A}(C) \in \overline{I} = I$
(since $I$ is integrally closed).
Reducing modulo $I$, we obtain $w_A\tup{C} = 0$
for every nontrivial cycle $C$.
Hence $A$ is nullcyclic.
\end{proof}

\begin{theorem}
\label{thm.integrally-closed-quotient-strong2}
Let $D$ be a commutative ring, let $I$ be an integrally closed ideal of
$D$, and set $R:=D/I$. Let $A$ be a $1$-principled matrix over $R$. Then $A$
is $1$-nullcyclic.
\end{theorem}

\begin{proof}
The matrix $A$ is $1$-principled.
By Proposition~\ref{prop.1-princ}, this entails that
$A$ is $1$-diagonal and principled.
Hence,
Theorem~\ref{thm.integrally-closed-quotient-strong} shows that
$A$ is nullcyclic. Since $A$ is $1$-diagonal, we thus conclude
that $A$ is $1$-nullcyclic.
\end{proof}

\begin{proof}[Proof of Theorem~\ref{thm.main2}]
By Theorem~\ref{thm.integrally-closed-quotient-strong2}, the matrix $A$ is
$1$-nullcyclic. Hence Corollary~\ref{cor.strong-integer-powers-principled2}
shows that $A^m$ is $1$-principled for every $m\in\mathbb Z$.
\end{proof}

\begin{proof}[Proof of Theorem~\ref{thm.main}]
By Theorem~\ref{thm.integrally-closed-quotient-strong}, the matrix $A$ is
nullcyclic. Hence Corollary~\ref{cor.strong-powers-principled}
shows that $A^m$ is principled for every $m\in\mathbb N$.
\end{proof}

\subsection{\texorpdfstring{The rings $\mathbb Z/d$}{The rings Z/d}}

The case of the ring $\mathbb Z/d = \ZZ / d \ZZ$ admits a particularly
elementary application of Theorem~\ref{thm.cycle-weight-integral}.
This relies on the following well-known fact, which we prove here
for the sake of completeness:

\begin{lemma}
\label{lem.Z-ideals-integrally-closed}
Every ideal of $\mathbb Z$ is integrally closed.
\end{lemma}

\begin{proof}
Every ideal of $\mathbb Z$ has the form $d\mathbb Z$ for some
$d\in\mathbb N$. The claim is clear for $d=0$, since the only nilpotent
integer is $0$. Assume that $d>0$, and let $x\in\mathbb Z$ be integral over
$d\mathbb Z$. Thus, for some positive integer $r$, we have
\begin{align}
x^r+c_1x^{r-1}+\cdots+c_rx^0=0
\qquad \text{with }c_j\in \tup{d\ZZ}^j = d^j\mathbb Z.
\label{pf.lem.Z-ideals-integrally-closed.1}
\end{align}
Write $c_j=d^jb_j$ with $b_j\in\mathbb Z$, and put $y=x/d\in\mathbb Q$.
Dividing the equation \eqref{pf.lem.Z-ideals-integrally-closed.1}
by $d^r$ gives
\begin{align}
y^r+b_1y^{r-1}+\cdots+b_ry^0=0.
\label{pf.lem.Z-ideals-integrally-closed.2}
\end{align}
Thus $y$ is a rational root of a monic polynomial in $\mathbb Z[X]$. To see
directly that $y\in\mathbb Z$, write $y=a/b$ in lowest terms with $b>0$.
After multiplication by $b^r$, the equation
\eqref{pf.lem.Z-ideals-integrally-closed.2} shows that $b\mid a^r$.
Since $\gcd(a,b)=1$, this forces $b=1$. Hence $y\in\mathbb Z$, so
$x=dy\in d\mathbb Z$.
\end{proof}

\begin{corollary}
\label{cor.zmod-n}
Let $d$ be an integer, and let $A$ be a $1$-principled matrix over
$\mathbb Z/d\mathbb Z$. Then $A^m$ is $1$-principled for every
$m\in\mathbb Z$.
\end{corollary}

\begin{proof}
The ideal $d \ZZ$ of $\ZZ$ is integrally closed by
Lemma~\ref{lem.Z-ideals-integrally-closed}.
Hence, Theorem~\ref{thm.main2}
(applied to $D = \ZZ$ and $I = d\ZZ$)
shows that $A^m$ is $1$-principled for every $m\in\mathbb Z$.
\end{proof}

\begin{remark}
Our above proof of Corollary~\ref{cor.zmod-n} is
self-contained.
Indeed, the only result we have used without proof is
Lemma~\ref{lem.integral-closure-basics}, which in this
case is being applied to $R = \ZZ$; but all claims of
this lemma
follow from Lemma~\ref{lem.Z-ideals-integrally-closed}
when $R = \ZZ$.
\end{remark}

The reader can easily state and prove an analogue of
Corollary~\ref{cor.zmod-n} with ``$1$-principled'' replaced
by ``principled''.

\subsection{Pr\"ufer domains}

We next give a broad class of domains for which every ideal is integrally
closed.

\begin{definition}
Let $D$ be an integral domain with fraction field $K$. A nonzero fractional
ideal $L$ of $D$ is called \emph{invertible} if there exists a fractional
ideal $M$ such that $LM=D$. The domain $D$ is called a \emph{Pr\"ufer domain}
if every nonzero finitely generated ideal of $D$ is invertible.
\end{definition}

For background on
Pr\"ufer domains and their equivalent characterizations, see
\cite[\S XII.3]{LomQui} and \cite{FontanaHuckabaPapick}.

\begin{lemma}
\label{lem.Prufer-ideals-integrally-closed}
Every ideal of a Pr\"ufer domain is integrally closed.
\end{lemma}

\begin{proof}
This is part of \cite[Theorem XII.3.2]{LomQui}, but we give
a proof for the sake of completeness.

Let $I$ be an ideal of a Pr\"ufer domain $D$, and let $x\in D$ be integral
over $I$. We must show that $x \in I$.

Choose an equation
\begin{align}
x^r+c_1x^{r-1}+\cdots+c_rx^0=0
\qquad \text{with } c_j\in I^j
\label{pf.Pruefer.eq1}
\end{align}
(since $x$ is integral over $I$).
Only finitely many elements of $I$ are needed to express all the $c_j$ as
sums of products of $j$ elements of $I$. Let $J\subseteq I$ be the finitely
generated ideal generated by these elements. Then $c_j\in J^j$ for every
$j$, so $x$ is integral over $J$.

Set $L:=J+xD$. If $J=0$, then \eqref{pf.Pruefer.eq1} gives $x^r=0$, whence
$x=0$ since $D$ is a domain; thus the claim is clear. Hence we may assume
that $J\neq0$. Then both $J$ and $L$ are nonzero finitely generated ideals,
and therefore invertible. Moreover, \eqref{pf.Pruefer.eq1} yields
\[
x^r\in JL^{r-1}.
\]
Any other product of $r$ generators of $L=J+xD$ already contains a factor
from $J$ and thus belongs to $JL^{r-1}$ as well.
Hence, $L^r \subseteq JL^{r-1}$. Since the opposite inclusion
is obvious, we thus have shown that
\[
L^r=JL^{r-1}.
\]
Since $L^{r-1}$ is invertible, we may cancel it and obtain $L=J$. In
particular, $x\in L=J\subseteq I$. Thus $I$ is integrally closed.
\end{proof}

\begin{corollary}
\label{cor.Prufer-quotient}
Let $D$ be a Pr\"ufer domain, let $I$ be an ideal of $D$, and let $A$
be a $1$-principled matrix over $D/I$. Then $A^m$ is $1$-principled for every
$m\in\mathbb Z$.
\end{corollary}

\begin{proof}
Lemma~\ref{lem.Prufer-ideals-integrally-closed} shows that $I$ is
integrally closed; thus, we can apply
Theorem~\ref{thm.main2}.
\end{proof}

The reader can easily state and prove an analogue of
Corollary~\ref{cor.Prufer-quotient} with ``$1$-principled'' replaced
by ``principled''.

\begin{remark}
Valuation domains are precisely the local Pr\"ufer domains
(see \cite[\S XII.3]{LomQui}). Thus quotients of
valuation domains are included in Corollary~\ref{cor.Prufer-quotient}.
\end{remark}

\subsection{Normal domains}

Recall that a normal domain is a domain that is integrally closed in its
fraction field. Every principal ideal of a normal domain is integrally closed;
see \cite[Proposition 1.5.2]{HunekeSwanson}. Therefore
Theorem~\ref{thm.main2} yields the following.

\begin{corollary}
\label{cor.normal-principal-quotient}
Let $D$ be a normal domain, let $f\in D$ be a nonzero nonunit, and let $A$ be a
$1$-principled matrix over $D/fD$. Then $A^m$ is $1$-principled for every
$m\in\mathbb Z$.
\end{corollary}


\end{document}